\documentclass[reqno]{amsart}

\usepackage[english]{babel}
\usepackage[utf8]{inputenc}
\usepackage{enumerate}
\usepackage{graphicx}
\usepackage{xcolor}

\usepackage{amsthm}
\usepackage{amsmath,amsfonts,amssymb}
\usepackage{latexsym,hyperref}
\usepackage{comment}
\usepackage[showonlyrefs]{mathtools}
\mathtoolsset{showonlyrefs=true}

\numberwithin{equation}{section}

\newtheorem{theorem}{Theorem}[section]
\newtheorem{corollary}[theorem]{Corollary}
\newtheorem{lemma}[theorem]{Lemma}
\newtheorem{proposition}[theorem]{Proposition}

\newtheorem{remark}[theorem]{Remark}

\newtheorem*{namedtheorem}{\theoremname}
  \newcommand{\theoremname}{testing}
  \newenvironment{named}[1]{
     \renewcommand{\theoremname}{#1}
     \begin{namedtheorem}}
     {\end{namedtheorem}}

\begin{document}
 
\title[Radial nodal solutions of Hénon type equations]{Morse index of radial nodal solutions of Hénon type equations in dimension two}

\thanks{E. Moreira dos Santos is partially supported by CNPq \#309291/2012-7 grant and FAPESP \#2014/03805-2 grant. F. Pacella is partially supported by PRIN 2009-WRJ3W7 grant and GNAMPA-INDAM}

\author{Ederson Moreira dos Santos}
\address{Ederson Moreira dos Santos \newline \indent Instituto de Ciências Matemáticas e de Computação --- Universidade de São Paulo \newline \indent
Caixa Postal 668, CEP 13560-970 - S\~ao Carlos - SP - Brazil}
\email{ederson@icmc.usp.br}

\author{Filomena Pacella}
\address{Filomena Pacella \newline \indent Dipartimento di Matematica --- Università di Roma {\it{Sapienza}} \newline \indent
P.le. Aldo Moro 2, 00184 Rome, Italy}
\email{pacella@mat.uniroma1.it}

\date{\today}
\subjclass[2010]{35B06; 35B07;  35J15; 35J61}
\keywords{Semilinear elliptic equations; Hénon equation; Nodal solutions; Morse index; Non-degeneracy}

\begin{abstract}
We consider non-autonomous semilinear elliptic equations of the type
\[
-\Delta u = |x|^{\alpha} f(u), \ \ x \in \Omega, \ \ u=0 \quad \text{on} \ \ \partial \Omega,
\]
where $\Omega \subset {\mathbb R}^2$ is either a ball or an annulus centered at the origin, $\alpha >0$ and $f: {\mathbb R}\ \rightarrow {\mathbb R}$ is $C^{1, \beta}$ on bounded sets of ${\mathbb R}$. We address the question of estimating the Morse index $m(u)$ of a sign changing radial solution $u$. We prove that $m(u) \geq 3$ for every $\alpha>0$ and that $m(u)\geq \alpha+ 3$ if $\alpha$ is even. If $f$ is superlinear the previous estimates become $m(u) \geq n(u)+2$ and $m(u) \geq \alpha+ n(u)+2$, respectively, where $n(u)$ denotes the number of nodal sets of $u$, i.e. of connected components of $\{ x\in \Omega; u(x) \neq 0\}$. Consequently, every least energy nodal solution $u_{\alpha}$ is not radially symmetric and $m(u_{\alpha}) \rightarrow + \infty$ as $\alpha \rightarrow + \infty$ along the sequence of even exponents $\alpha$.
\end{abstract}
\maketitle

\section{Introduction}
Let us consider a non-autonomous semilinear elliptic equation of the type
\begin{equation}\label{generalf}
-\Delta u  = g(|x|, u) \ \ \text{in} \ \ \Omega, \quad u = 0 \ \ \text{on} \ \ \partial \Omega,
\end{equation}
where $\Omega \subset {\mathbb R}^N$, $N \geq 2$, is either a ball or an annulus centered at the origin, $g: [0, +\infty) \times {\mathbb R} \rightarrow {\mathbb R}$ is such that $r \mapsto g(r,u)$ is $C^{0, \beta}$ on bounded sets of $[0,+\infty) \times {\mathbb R}$, $u \mapsto g_u(r,u)$ is  $C^{0, \gamma}$ on bounded sets of $[0, + \infty) \times {\mathbb R}$, where $g_u$ denotes the derivative of $g$ with respect to the variable $u$. Since the problem is invariant by spherical symmetry we can consider classical radial solutions of \eqref{generalf}. Here we address the question of estimating the Morse index of sign changing radial solutions of \eqref{generalf}.

Given any continuous function $u:\Omega \rightarrow {\mathbb R}$ we will denote by $n(u)$ the number of nodal sets of $u$, i.e. of connected components of $\{ x\in \Omega; u(x) \neq 0\}$.

We recall that the Morse index $m(u)$ of a solution $u$ of \eqref{generalf} is the maximal dimension of a subspace of $H^1_0(\Omega)$ in which the quadratic form
\begin{equation}\label{Qint}
w \longmapsto Q_u(w,w) = \int_{\Omega} |\nabla w (x)|^2 dx  - \int_{\Omega} g_u(|x|, u(x)) w^2(x) dx
\end{equation}
is negative definite. Alternatively, since we are considering the case of bounded domains, $m(u)$ can be defined as the number of negative eigenvalues, counted with their multiplicity, of the linearized operator $L_u : = -\Delta - g_u(|x|, u)$ in the space $H^1_0(\Omega)$.

In the case of autonomous problems, i.e. when the nonlinear term $g$ does not depend on the space variable, Aftalion and Pacella \cite{AftalionPacella}, as a consequence of a more general result in symmetric domains, obtained the following theorem.

\begin{named}{Theorem A}[Autonomous problems]\label{ThA}
Let $g(r,u) = f(u)$ with $f \in C^1({\mathbb R})$. Then any sign changing radial solution of \eqref{generalf} has Morse index greater than or equal to $N+1$.
\end{named}

\begin{remark}\label{remark4int}
More precisely in \cite{AftalionPacella} it is proved that the linearized operator $L_u$ has at least $N$ negative eigenvalues whose corresponding eigenfunctions are non-radial and change sign. Therefore, adding the first eigenvalue, which is obviously associated to a radial eigenfunction, one gets at least $N+1$ negative eigenvalues. In the case when $f$ is superlinear, i.e. satisfies \eqref{H:superlinear}, then it is easy to see, testing the quadratic form  on the solution $u$ in each nodal region, that there are at least $n(u)$ negative eigenvalues in the space of radial functions. Hence for these nonlinearities, any sign changing radial solution has Morse index greater than or equal to $N+n(u)$. In particular this holds for Lane-Emden problems, i.e.
\begin{equation}\label{p>2}
-\Delta u  = |u|^{p-1}u \ \ \text{in} \ \ \Omega, \quad u = 0 \ \ \text{on} \ \ \partial \Omega, \ \ p>1.
\end{equation}
We also point out that the assumption $f(0) \geq 0$ in \cite{AftalionPacella} is not really needed.
\medbreak
\end{remark}

As a consequence of Theorem A and in the case of superlinear, subcritical problems, like \eqref{p>2} for $p <\frac{N+2}{N-2}$ if $N \geq 3$, in \cite{AftalionPacella} it is deduced that any least energy nodal solution cannot be radial, since their Morse index is precisely $2$; cf. \cite{CastroCossioNeuberger, BartschChangWang, BartschWeth}. Obviously this break of symmetry is relevant for many applications.

The proof of Theorem A uses in a crucial way the fact that the derivatives $\frac{\partial u}{\partial x_i}$, $i = 1, \ldots, N$, of a solution $u$ of \eqref{generalf} are indeed solutions of the linearized equation $L_u (w) = 0$. This property is a peculiarity of autonomous problems. For this reason the proof of \cite{AftalionPacella} does not extend to the case of non-autonomous nonlinearities. So it is an open question to understand whether a similar estimate on the Morse index of nodal radial solutions holds  for the general problem \eqref{generalf} and also whether least energy nodal solutions are radial or not.  

In this paper we answer these questions in the case of nonlinearities of the type $g(|x|, u) = |x|^{\alpha} f(u)$ and $N=2$. More precisely we consider the problem
\begin{equation}\label{weight+f}
-\Delta u  = |x|^{\alpha} f(u) \ \ \text{in} \ \ \Omega, \quad u = 0 \ \ \text{on} \ \ \partial \Omega,
\end{equation}
where $\alpha> 0$, $\Omega \subset {\mathbb R}^2$ is either a ball or an annulus centered at the origin and $f: {\mathbb R} \rightarrow {\mathbb R}$ is $C^{1, \beta}$ on bounded sets of ${\mathbb R}$. In some of our results we also assume the following superlinear condition
\begin{equation}\label{H:superlinear}
f'(u) > \frac{f(u)}{u}\quad \forall \, u \in {\mathbb R}\backslash\{0\}.
\end{equation}

Our first result is the following.

\begin{theorem}\label{maintheoremN=2}
Let $u$ be a radial sign changing solution of \eqref{weight+f}. Then $u$ has Morse index greater than or equal to $3$. Moreover, if \eqref{H:superlinear} holds, then the Morse index of $u$ is at least $n(u)+2$. 
\end{theorem}

In the case that $f(u) = |u|^{p-1}u$, with $p>1$,  \eqref{weight+f} turns out to be the so called Hénon equation \cite{Henon}
\begin{equation}\label{henon}
-\Delta u = |x|^{\alpha}|u|^{p-1}u \quad x \in \Omega, \qquad u=0 \quad \text{on} \quad \partial \Omega, 
\end{equation}
which has been extensively studied since the work of Ni \cite{Ni}. We mention that apart from its mathematical interest, the Hénon equation is important in the \linebreak applications, in particular in astrophysics; cf. \cite{Henon, EdersonPacella}. Ni considered \eqref{henon} in the case of $\Omega$ being an open ball centered at zero in ${\mathbb R}^N$ with $N\geq 3$. In this case the Poho{\v{z}}aev identity, as in \cite[Lemma 1.1]{deFigueiredoLionsNussbaum}, shows that \eqref{henon} has no nontrivial solution if $p \geq \frac{N+2 + 2\alpha}{N-2}$. On the other side, with $1 < p < \frac{N+ 2 + 2\alpha}{N-2}$, the existence of a positive radial solution can be proved by using classical variational methods, for example, combining the Radial Lemma in \cite{Ni} with the mountain pass theorem.  Again in the same range of $p$, a combination of the Radial Lemma in \cite{Ni} with some arguments in \cite{BartschWeth} gives the existence of a least energy solution among the nodal radial solutions of \eqref{henon}, hereafter called least energy nodal radial solution. In addition, in the case when $\Omega$ is an annulus, these existence results hold trivially for any $p>1$, since no lack of compactness occurs in the setting of radial functions. 

Next we recall that it is proved in \cite[Theorem 1.3]{BartschWeth}, see also \cite{CastroCossioNeuberger}, that a least energy nodal solution of \eqref{weight+f} exists and has Morse index $2$ if $f$ satisfies \eqref{H:superlinear} and the additional conditions:
\begin{equation}\label{growth}
f(0) = 0 \quad \text{and} \quad \exists \, p>1 \ \ s.t. \ \ |f'(u)| \leq C (1 + |u|^{p-1}) \ \ \forall \, u \in {\mathbb R},
\end{equation}
\begin{equation}\label{A-R}
\exists \, R> 0, \ \ \theta >2 \ \ s.t. \ \ 0 < \theta \int_0^u f(\tau) d\tau \leq u f(u) \ \ \forall \, |u| \geq R.
\end{equation}
Then, as a consequence of Theorem \ref{maintheoremN=2}, we get the following result.

\begin{corollary}\label{cor:symmetrybreaking}
Assume \eqref{H:superlinear}, \eqref{growth} and \eqref{A-R}. Then any least energy nodal solution of \eqref{weight+f} is not radially symmetric.
\end{corollary}

In contrast to the above symmetry breaking result, we recall that it is proved in \cite{PacellaWeth, BartschWethWillem} that every least energy nodal solution of \eqref{weight+f} is foliated Schwarz symmetric, i.e. axially symmetric and monotone in the angular coordinate. We also point out that Corollary \ref{cor:symmetrybreaking} was already shown for the Hénon equation \eqref{henon}, for every $N\geq2$, but only for particular cases of $\alpha$: for $\alpha$ large in \cite[Remark 6.4]{BartschWethWillem} by a comparison of energy argument; for $\alpha$ small  in \cite[Corollary 1.6 (iii)]{BonheuredosSantosRamosTavares} by an asymptotic analysis, as $\alpha \rightarrow 0$, of the least energy nodal solutions. The general symmetry breaking result, for any $\alpha> 0$, was, up to now, an open question.

We point out that the proof of Theorem \ref{maintheoremN=2} is different from that of Theorem A of \cite{AftalionPacella}. Indeed it relies on a suitable change of variable which works well in ${\mathbb R}^2$. This change of variable was considered previously in \cite{ClementdeFigueiredoMitidieri}, see also the recent papers \cite{CowanGhoussoub, GladialiGrossiNeves}, where an alternative approach to identify the critical exponent $\frac{N+2 + 2\alpha}{N-2}$, $N \geq 3$, associated with the Hénon equation \eqref{henon} in the case when $\Omega$ is an open ball centered at zero in ${\mathbb R}^N$, was presented. In these three papers, while studying radial solutions, the authors consider the corresponding ODE problem. Then, the critical exponent $\frac{N+2 + 2\alpha}{N-2}$ comes out as a result of a suitable one dimensional change of variable that reduces the weighted problem to a non-weighted one. 

The novelty in our arguments consists in applying the change of variable to functions in ${\mathbb R}^2$ which are not necessarily radially symmetric, even though it does not act well on the gradient or on the Laplacian as it does for spherically symmetric functions; cf. \eqref{fundamentalgradientspolar}, Remark \ref{changelaplacian}, \eqref{fundamentalgradientsradial} and \eqref{laplacians}. Nevertheless, we show that it is useful to get an estimate from below on the Morse index of radial nodal solutions of \eqref{weight+f} in the whole space $H^1_0(\Omega)$, i.e. not only on radial directions; cf. Proposition  \ref{comparisonlemma}.

Another question which arises from Theorem \ref{maintheoremN=2} is that of having a more precise estimate on the Morse index as the exponent $\alpha$ varies. How does the weight $|x|^{\alpha}$ influence the Morse index of nodal radial solutions of \eqref{weight+f}\,? In this direction, using some different changes of variables, we prove that the Morse indices go to infinity along the sequence of even exponents $\alpha$.

\begin{theorem}\label{theoremalpha=2}
Let $\alpha>0$ be even and let $u$ be a radial nodal solution of \eqref{weight+f}. Then $u$ has Morse index greater than or equal to $\alpha +3$. If in addition \eqref{H:superlinear} holds, then the Morse index of $u$ is at least $ n(u) + \alpha + 2$. 
\end{theorem}

The proof of Theorem \ref{theoremalpha=2} relies on a modification of the previous change of variable that works fine for the case when $\alpha$ is even. This change of variable is the key argument to prove the existence of many negative eigenvalues of the linearized operator $L_u$, associated to a radial sign changing solution $u$ of \eqref{weight+f}, and related to the weighted problem
\begin{equation}\label{weightedeigenvalue}
-\Delta \varphi - |x|^{\alpha} f'(u) \varphi  = \lambda |x|^{\alpha} \varphi \ \ \text{in} \ \ \Omega, \ \ \varphi = 0 \ \ \text{on} \ \ \partial \Omega.
\end{equation}
Indeed its peculiarity is to transform eigenfunctions of the non-weighted problem \eqref{4.3''} with a certain symmetry into eigenfunctions of \eqref{weightedeigenvalue} with a different symmetry. A variant of this was used in \cite{PacellaSrikanth} in higher dimensions to pass from doubly symmetric solutions of a supercritical problem in dimension $2m$, $m \geq 2$, to axially symmetric solutions of a subcritical problem in dimension $m+1$. Here we do not change dimension but we apply a somehow similar idea to create a correspondence between eigenfunctions of linearized operators of two different problems. We believe that the simple ideas exploited in this paper could be useful in other kind of problems.

Next we consider the particular case of the Hénon equation \eqref{henon} and we prove the following non-degeneracy result.
\begin{theorem}\label{theorem:radialsolutions}
Let $\alpha \geq 0$ and $p>1$.
\begin{enumerate}[i)]
\item For each $n \in {\mathbb N}$ there is only one radial solutions $u_{\alpha,n}$ of \eqref{henon}, up to multiplication by $-1$, with $n$ nodal sets. Moreover,
\[
u_{\alpha,n} (x) = \left(\frac{\alpha+2}{2}\right)^{\frac{2}{p-1}} U_{\alpha,n}(|x|^{\frac{\alpha}{2}}x)
\]
where $U_{\alpha,n}$ is the unique, up to multiplication by $-1$, nodal radial solution of \eqref{p>2} in $\Omega_{\alpha}= \{ |x|^{\frac{\alpha}{2}}x; \, x \in \Omega \}$ with $n$ nodal regions.
\vspace{5pt}

\item Let $u_{\alpha}$ be a least energy nodal radial solution of \eqref{henon}. Then $u_{\alpha}$ has two nodal regions, and so $u_{\alpha} = u_{\alpha,2}$ or $u_{\alpha,p} = -u_{\alpha,2}$. Moreover, it is non-degenerate in the space of radial functions, that is, if $\varphi$ is a radial solution of
\[
-\Delta \varphi = p |x|^{\alpha}|u_{\alpha}|^{p-1} \varphi \ \ \text{in} \ \ \Omega, \quad \varphi = 0 \ \ \text{on} \ \ \partial \Omega,
\]
then $\varphi \equiv 0$.
\end{enumerate}
\medbreak
\end{theorem}
Finally, consider the case when $\Omega$ is the unit ball in ${\mathbb R}^2$ centered at zero. Then $\Omega_{\alpha} = \Omega$ for all $\alpha >0$ and $U_{\alpha,2}$ does not depend on $\alpha$ as well, hence we denote $U_{\alpha, 2}$ simply by $U$. Then the non-degeneracy of $u_{\alpha}$ in $H^1_{0, {\rm{rad}}}(\Omega)$, i.e. ii) of Theorem \ref{theorem:radialsolutions}, together with Theorem  \ref{theoremalpha=2}, i.e. $m(u_{\alpha}) \rightarrow + \infty$ along the sequence of even exponents $\alpha$, indicates that there should be infinitely many branches of non-radial solutions of \eqref{henon} bifurcating  from the curve
\[
C = \left\{ u_{\alpha}: \alpha> 0, \ u_{\alpha} (x) = \left(\frac{\alpha+2}{2}\right)^{\frac{2}{p-1}} U(|x|^{\frac{\alpha}{2}}x)\right\}
\]
of least energy nodal radial solutions of \eqref{henon}. 

This paper is organized as follows. In Section \ref{section:generalchangeofvariables} we introduce a change of variable in ${\mathbb R}^2$, we prove several properties of it and Theorem \ref{theorem:radialsolutions}. Then in Section \ref{section:proofmaintheorem}, based on the results from Section \ref{section:generalchangeofvariables}, we compare the Morse indices of radial nodal solutions of \eqref{weight+f} with those of the corresponding nodal solutions of a non-weighted problem, and we prove Theorem \ref{maintheoremN=2}. Finally, in Section \ref{section:doublysymmetric}, in the case of even $\alpha$, we consider some slightly different changes of variables in ${\mathbb R}^2$ which again relate weighted semilinear elliptic equations like \eqref{weight+f}  to corresponding non-weighted ones. This allows to produce more directions in which the quadratic form $Q_u$ is negative definite proving so Theorem \ref{theoremalpha=2}.

\pagebreak

\section{Preliminary results} \label{section:generalchangeofvariables}

\subsection{A useful change of variable}
Let us fix some notation that will be used throughout in this paper. To a point $x= (x_1, x_2) \in {\mathbb R}^2$ in cartesian coordinates, we will associate the polar coordinates $(r, \theta)$, namely
\begin{equation}\label{identification x}
x_1 = r \cos{\theta}, \ \ x_2 = r \sin{\theta},  \ \ r= \sqrt{|x_1|^2 + |x_2|^2}.
 \end{equation}
So, for every function $u$ defined according to the cartesian coordinates $(x_1, x_2)$, we will write
\begin{equation}\label{identification u}
u(x_1, x_2) = u(r \cos \theta, r \sin \theta) = u(r, \theta). 
\end{equation} 
Then we recall the following formulae
\begin{equation}
\nabla_ x = \left(\frac{\partial}{\partial x_1}, \, \frac{\partial}{\partial x_2}\right) = \left( \cos \theta \frac{\partial}{\partial r} - \frac{1}{r} \sin \theta \frac{\partial}{\partial \theta}, \, \sin \theta \frac{\partial}{\partial r} + \frac{1}{r} \cos \theta \frac{\partial}{\partial \theta}\right),
\end{equation}
\begin{equation}\label{gradientlaplacianpolar}
|\nabla_ x|^2 = \left( \frac{\partial}{\partial x_1} \right)^2 + \left( \frac{\partial}{\partial x_2} \right)^2 = \left(  \frac{\partial}{\partial r}\right)^2 + \frac{1}{r^2} \left( \frac{\partial}{\partial \theta} \right)^2,
\end{equation}
and
\begin{equation}\label{laplacianpolar}
\Delta_x = \frac{\partial^2}{\partial x_1^2} + \frac{\partial^2}{\partial x_2^2} =  \frac{\partial^2}{\partial r^2} + \frac{1}{r} \frac{\partial}{\partial r} + \frac{1}{r^2} \frac{\partial^2}{\partial \theta^2}.
\end{equation} 
 
We will perform some changes of variables $x \longleftrightarrow y$ in ${\mathbb R}^2$. Then to $y = (y_1, y_2) \in {\mathbb R}^2$ we will associate the polar coordinates $(s, \sigma)$ by setting
\begin{equation}\label{identification y}
y_1 = s \cos{\sigma},  \ \ y_2 = r \sin{\sigma},  \ \ s= \sqrt{|y_1|^2 + |y_2|^2}.
 \end{equation}
As before, if the function $v$ is defined according to the cartesian coordinates $(y_1, y_2)$ then we will also write
 \[
 v(y_1, y_2) = v(s \cos \sigma, s \sin \sigma) = v(s, \sigma). 
 \]

Let $\kappa>0$ and consider the following transformation
\begin{equation}\label{tkappa}
T_{\kappa}: {\mathbb R}^2 \rightarrow {\mathbb R}^2,  \quad T_{\kappa} y := y |y|^{\kappa-1},
\end{equation}
where we set $T_{\kappa}(0,0):=(0,0)$ and $x = T_{\kappa}y$. Then, with respect to the polar coordinates $(s, \sigma)$ and $(r, \theta)$, the transformation $T_{\kappa}$ reads
\begin{equation}\label{tkappapolar}
T_{\kappa}: {\mathbb R}^2 \rightarrow {\mathbb R}^2,  \quad T_{\kappa} (s, \sigma) := (s^{\kappa}, \sigma), \ \ \text{i.e.}, \ \ r = s^{\kappa}, \ \ \theta = \sigma.
\end{equation}

The transformation $T_{\kappa}$ has a simpler expression in polar coordinates, which shortens many computations. In view of the applications, we present some of our results, and arguments, also in cartesian coordinates. 

\begin{lemma}\label{lemma:tkappa}
The following properties hold.
\begin{enumerate}[i)]
\item $T_{\kappa}$ is a homeomorphism whose inverse is 
\begin{equation}\label{inversetkappa}
T_{\kappa}^{-1}x = x |x|^{\frac{1}{\kappa}-1}, \ \ \text{i.e.}, \ \ T_{\kappa}^{-1} = T_{\frac{1}{\kappa}}.
\end{equation}
\item In cartesian coordinates, the Jacobian matrix of $T_{\kappa}$ is
\begin{equation}
J_{T_{\kappa}}(y) = \frac{\partial (x_1, x_2)}{\partial (y_1, y_2)} (y) = |y|^{\kappa-3}
\left[
\begin{array}{cc}
|y|^2 + (\kappa -1) y_1^2 & (\kappa -1) y_1y_2\\ \\
(\kappa -1) y_1y_2 & |y|^2 + (\kappa -1) y_2^2
\end{array}
\right], \quad \forall \ y \neq0 
\end{equation}
and
\begin{equation}\label{jacobian}
\left| det\, J_{T_{\kappa}}(y) \right| = \kappa \, |y|^{2\kappa  - 2}, \quad \forall \ y \neq0.
\end{equation}

\item Given a function $\psi$ defined on a subset of ${\mathbb R}^2$, set $\varphi = \psi \circ T_{\kappa}^{-1}$. Let $y \neq 0$ and $x = T_{\kappa}y$. Then $\psi$ is differentiable at $y$ if and only if $\varphi$ is differentiable at $x$. 

\item Let $\psi$, $\varphi$, $y$, $x$ as before and $r$, $s$, $\sigma$ and $\theta$ as in \eqref{tkappapolar}. Then
\begin{equation}\label{fundamentalgradientspolar}
\left[\psi_s^2 + \frac{1}{s^2} \psi_{\sigma}^2\right] s^{2 -2\kappa}= k^2 \varphi_ r^2 + \frac{1}{r^2} \varphi_{\theta}^2, \ \ \forall \, s \neq0,
\end{equation}
which implies that
\begin{equation}\label{fundamentalgradients}
\min\{1,\kappa^2\} | \nabla \varphi (x)|^2 \leq |\nabla \psi (y)|^2 \, |y|^{2 -2\kappa}\leq  \max\{1, \kappa^2\}| \nabla \varphi (x)|^2, \ \ \forall \ y \neq 0.
\end{equation}
Moreover, if $\psi$ is radially symmetric, then
\begin{equation}\label{fundamentalgradientsradial}
\kappa^2 | \nabla \varphi (x)|^2  = |\nabla \psi (y)|^2 \, |y|^{2 -2\kappa}, \ \ \forall \ y \neq 0.
\end{equation}
\end{enumerate}
\end{lemma}
\begin{proof}
The statements from i), ii) and iii) are just matter of computation. Regarding iv), the identity \eqref{fundamentalgradientspolar} follows from \eqref{tkappapolar}. From \eqref{fundamentalgradientspolar} we infer that
\[
\min\{ 1, k^2 \} \left[ \varphi_ r^2 + \frac{1}{r^2} \varphi_{\theta}^2 \right] \leq  \left[\psi_s^2 + \frac{1}{s^2} \psi_{\sigma}^2\right]s^{2 -2\kappa}  \leq \max\{ 1, k^2 \} \left[ \varphi_ r^2 + \frac{1}{r^2} \varphi_{\theta}^2 \right]
\]
which combined with \eqref{gradientlaplacianpolar} implies \eqref{fundamentalgradients}. If $\psi$ is radially symmetric, it is also clear that \eqref{fundamentalgradientsradial} follows from \eqref{fundamentalgradientspolar} since $\psi_{\sigma} \equiv0$ and $\varphi_{\theta} \equiv 0$.
\end{proof}

From now on  in this section $\Omega \subset {\mathbb R}^2$ represents either a ball or an annulus centered at the origin and  we set $\Omega_{\kappa} = T^{-1}_{\kappa}(\Omega)$, where $T_{\kappa}$ is given by \eqref{tkappa}.

\begin{lemma} \label{lemma:lr}
Let $1 \leq r < \infty$. Then
\[
S_{\kappa}: L^r (\Omega_{\kappa}) \rightarrow L^r(\Omega, |x|^{\frac{2- 2\kappa}{\kappa}}), \ \ \text{defined by} \ \  S_{\kappa} \psi := \psi \circ T^{-1}_{\kappa},
\]
is a continuous linear isomorphism such that
\begin{equation}\label{lrnorms}
\int_{\Omega_{\kappa}} |\psi(y)|^r dy = \kappa^{-1} \int_{\Omega} |\varphi (x)|^r |x|^{\frac{2 - 2 \kappa}{\kappa}} dx, \ \ \text{with} \ \ \varphi = \psi \circ T^{-1}_{\kappa}.
\end{equation}
\end{lemma}
\begin{proof}
In the case when $\Omega$ is an annulus centered at the origin, then \eqref{lrnorms} comes out as an application of the standard change of variables theorem, using \eqref{inversetkappa} and \eqref{jacobian}.

In the case when $\Omega = B(0,R)$ is a ball centered at the origin and radius $R>0$, the singularity at zero of $T_{\kappa}$ or $T_{\kappa}^{-1}$ causes no problem, since we can reduce the arguments to the previous case by approximation with annuli. Indeed, take into account that
\[
\int_{B(0,R)} |h(z)| dz = \lim_{\delta \rightarrow 0^+} \int_{B(0,R) \backslash B(0, \delta)}|h(z)| dz, \quad \forall \ h \in L^1(B(0,R)).
\]
Then the monotone convergence theorem, passing to the limit, gives the result for the ball.
\end{proof}

With the same arguments we can prove the following lemma.

\begin{lemma} \label{lemma:flrnorms}
Let $F:{\mathbb R} \rightarrow {\mathbb R}$ be a continuous function. Then $F \circ \psi \in L^1(\Omega_{\kappa})$ if, and only if, $F \circ \varphi  \in L^1(\Omega, |x|^{\frac{2- 2\kappa}{\kappa}})$ with $\varphi = \psi \circ T^{-1}_{\kappa}$. Moreover,
\begin{equation}\label{compositionf}
\int_{\Omega_{\kappa}} F(\psi(y)) dy = \kappa^{-1} \int_{\Omega} F(\varphi (x)) |x|^{\frac{2 - 2 \kappa}{\kappa}} dx.
\end{equation}
\end{lemma}
We point out that if $\kappa = \frac{2}{\alpha +2}$, then $\frac{2- 2\kappa}{\kappa}= \alpha$ and so the weights $|x|^{\frac{2 - 2 \kappa}{\kappa}}$ at \eqref{compositionf} and $|x|^{\alpha}$ at \eqref{weight+f} coincide.

\begin{lemma} \label{lemma:h1}
The application
\[
S_{\kappa}: H^1_0 (\Omega_{\kappa}) \rightarrow H^1_0(\Omega), \ \ \text{defined by} \ \  S_{\kappa} \psi := \psi \circ T^{-1}_{\kappa},
\]
is a continuous linear isomorphism. Moreover, with  $\varphi = \psi \circ T^{-1}_{\kappa}$,
\begin{equation}\label{h1norms}
\min\left\{ \kappa,\frac{1}{\kappa}\right\} \int_{\Omega} |\nabla \varphi (x)|^2 dx \leq \int_{\Omega_{\kappa}} |\nabla \psi(y)|^2 dy \leq \max\left\{ \kappa,\frac{1}{\kappa}\right\} \int_{\Omega} |\nabla \varphi (x)|^2 dx, 
\end{equation}
for all  $\psi \in H^1_{0}(\Omega_{\kappa})$ and
\begin{equation}\label{h1norms=}
\kappa \int_{\Omega} |\nabla \varphi (x)|^2 dx = \int_{\Omega_{\kappa}} |\nabla \psi(y)|^2 dy,  \ \ \forall \ \psi \in H^1_{0,{\rm{rad}}}(\Omega_{\kappa}).
\end{equation}
\end{lemma}
\begin{proof}
Here we use \eqref{inversetkappa}, \eqref{jacobian}, \eqref{fundamentalgradients}, \eqref{fundamentalgradientsradial} and we proceed as in the proof of Lemma \ref{lemma:lr}.
\end{proof}

\begin{remark}\label{remarkN>=3H1}
Let $N\geq 3$, $\kappa>0$ and consider the homeomorphism $T_{\kappa}: {\mathbb R}^N \rightarrow {\mathbb R}^N$ defined by
\[
T_{\kappa}(y_1, \ldots, y_N) = (y_1, \ldots, y_N)|(y_1, \ldots, y_N)|^{\kappa-1}
\]
i.e. the same as \eqref{tkappa} but in ${\mathbb R}^N$. Then observe that a result like the one of Lemma \ref{lemma:h1} cannot hold. For example, consider $\Omega = B(0,1)$ and $\psi (y) = |y|^{-\gamma}-1$, with $0< \gamma< \frac{N-2}{2}$ and $0 < \kappa \leq \frac{2\gamma}{N-2}$. Then, under these conditions, $\psi \in H^1_0(\Omega_{\kappa})$ but $\psi \circ T_{\kappa}^{-1} \notin H^1_0(\Omega)$.
\end{remark}

\subsection{Equivalence between some weighted and non-weighted elliptic equations in the setting of radial solutions} \label{section:radial}
Hereafter in this section we consider the change of variable \eqref{tkappa} restricted to radial functions. In this setting it was already used in \cite{ClementdeFigueiredoMitidieri, CowanGhoussoub, GladialiGrossiNeves}.

Let $\Omega \subset {\mathbb R}^2$ be either a ball or an annulus centered at the origin and set $\Omega_{\kappa} = T^{-1}_{\kappa}(\Omega)$, where $T_{\kappa}$ is given by \eqref{tkappa}. For a radial function $u: \Omega \subset {\mathbb R}^2 \rightarrow {\mathbb R}$ we define the radial function $v: \Omega_{\kappa} \rightarrow {\mathbb R}$ by setting $v(y) = u(T_{\kappa}y)$, i.e.,
\begin{equation}\label{change1}
v(s) = u(s^k) = u(r), \quad r = s^{\kappa}, \ r = |x|, \ s = |y|.
\end{equation}
Then an easy computation yields
\begin{equation}\label{eqnova}
v_{ss}(s) + \frac{1}{s} v_{s}(s) = \kappa^2s^{2\kappa -2} \left[ u_{rr}(s^{\kappa}) + \frac{1}{s^{\kappa}} u_r(s^{\kappa}) \right], \ \  s>0.
\end{equation}
So, using the previous notation in polar coordinates,  we infer that
\begin{equation}\label{laplacians}
\Delta v (y)  = \kappa^2 |y|^{2\kappa -2} \Delta u (T_{\kappa}y) = \kappa^2 |x|^{2 -\frac{2}{\kappa}} \Delta u (x), \ \ r = |x|,  \ \ s = |y|, \ \ r = s^{\kappa}.
\end{equation}
Hence, if $u$ is a radial solution of the Hénon type equation \eqref{weight+f}, then $v: \Omega_{\kappa} \rightarrow {\mathbb R}$ is a radial function that satisfies
\[
-\Delta v(y) = \kappa^2|y|^{2\kappa -2 + \kappa \alpha} f(v(y)),  \quad y \in \Omega_{\kappa}, \quad v =0 \quad \text{on} \quad \partial \Omega_{\kappa}.
\]
Thus if we choose $\kappa$ such that
\begin{equation}\label{kappabom}
2\kappa -2 + \kappa \alpha = 0, \quad \text{i.e.}, \quad \kappa = \frac{2}{\alpha+2},
\end{equation}
then we infer that
\begin{equation}\label{reduced}
-\Delta v (y) = \left(\frac{2}{\alpha+2}\right)^2 f(v(y)), \quad y \in \Omega_{\kappa}, \quad v = 0 \quad \text{on} \quad \partial \Omega_{\kappa}.
\end{equation}

\begin{remark}\label{changelaplacian}
It is clear that, in general, the change of variable \eqref{tkappa} does not satisfy
\[
\Delta_{y} = \kappa^2 |x|^{2- \frac{2}{\kappa}}\Delta_x,
\]
as it does for radial functions; cf. \eqref{laplacians}. Indeed from \eqref{laplacianpolar} it is evident that also the angular part should be taken into account to write the complete Laplacian. However, see Proposition \ref{comparisonlemma}, the change of variable \eqref{tkappa}, with $\kappa = \frac{2}{\alpha+2}$, turns out to be very useful to compare the Morse index of a radial solution $u$ of \eqref{weight+f} and the Morse index of the corresponding radial solution $v = u \circ T_{\kappa}$ of \eqref{reduced}.
\end{remark}

\begin{remark}\label{doesnotworkN>=3}
Let $N\geq 3$ and $\alpha >0$. Then it is easy to see, just a matter of computation as in \cite[Proposition 4.2]{CowanGhoussoub}, that it is not possible to find a one dimensional change of variable
\[
r = s^{\kappa}, \quad r = |x|,  \quad s = |y|, \quad x, y \in {\mathbb R}^N,
\]
that is, to find  $\kappa$, such that
\[
v(s) = u(s^{\kappa}) \ \ \text{and} \ \ \Delta_{y} v = C \frac{1}{|x|^{\alpha}}\Delta_x u, \ \ C \ \ \text{constant},
\]
in the setting of radial functions defined in ${\mathbb R}^N$. This is one of the reasons why the proofs of this paper cannot be extended to dimension $3$ or higher.
\end{remark}

\begin{proof}[\textbf{Proof of Theorem \ref{theorem:radialsolutions}}]
${}$

\noindent i) This can be deduced by the analogous result for Lane-Emden equation \eqref{p>2}, cf. \cite[Theorem 2.15]{Ni1983} and \cite[p. 263]{Kajikiya1990}, by using the transformation \eqref{change1} and the identitie \eqref{laplacians} with $\kappa = \frac{2}{\alpha+2}$.

\medbreak

\noindent ii)  Let $u_{\alpha}$ be a least energy nodal radial solution of \eqref{henon}. Since the Morse index of $u_{\alpha}$ in $H^1_{0, {\rm{rad}}}(\Omega)$ is two, then $u_{\alpha}$ has precisely two nodal regions and then $u_{\alpha} (x) = \left(\frac{\alpha+2}{2}\right)^{\frac{2}{p-1}} U_{\alpha}(|x|^{\frac{\alpha}{2}}x)$ where $U_{\alpha}$, up to multiplication by $-1$, is the unique least energy nodal radial solution of \eqref{p>2} in $\Omega_{\alpha}= \{ |x|^{\frac{\alpha}{2}}x; \, x \in \Omega \}$. Moreover, the equation \eqref{laplacians} with $\kappa = \frac{2}{\alpha+2}$ guarantees that $u_{\alpha}$ is a degenerate radial solution of \eqref{henon} in the space $H^1_{0, {\rm{rad}}}(\Omega)$ if, and only if, $U_{\alpha}$ is a degenerate radial solution of the Lane-Emden equation \eqref{p>2} in $\Omega_{\alpha}$ in the space $H^1_{0, {\rm{rad}}}(\Omega_{\alpha})$.

So the above argument reduces the proof to the case $\alpha=0$, i.e. to the Lane-Emden equation. With $\alpha=0$ and in the case that $\Omega$ is an annulus, this non-degeneracy result is known; cf. \cite[Proposition 4]{PacellaSalazar}. Next, essentially, we mimic the arguments from \cite[Proposition 4]{PacellaSalazar} to include both cases of a ball and an annulus.

Let $u$ be a least energy nodal radial solution of \eqref{p>2}. We know that $u$ has precisely two nodal sets and Morse index $2$ in the space $H^1_{0, {\rm{rad}}}(\Omega)$. By contradiction, suppose that $u$ is degenerate in $H^1_{0, {\rm{rad}}}(\Omega)$. Then the third eigenvalue in the space $H^1_{0, {\rm{rad}}}(\Omega)$ of $L_u = -\Delta - p|u|^{p-1}$ is zero, and hence there exists $w$, a radial solution of 
\begin{equation}\label{zeroeigenvalue}
 - \Delta w = p |u|^{p-1} w \ \ \text{in} \ \ \Omega, \quad w = 0 \ \ \quad{on} \ \ \partial \Omega,
\end{equation}
with precisely three nodal regions. Now consider the auxiliary function
\[
z = x \cdot \nabla u + \frac{2}{p-1} u.
\]
Then, by direct computation, we obtain that
\begin{equation}\label{auxiliaryz}
 -\Delta z = p |u|^{p-1}z \ \ \text{in} \ \ \Omega, \quad z(x) = x \cdot \nabla u (x), \ \ x \in \partial \Omega.
\end{equation}
Next we multiply \eqref{zeroeigenvalue} by $z$, \eqref{auxiliaryz} by $w$ and we integrate by parts. The two resulting identities yield
\[
\int_{\partial \Omega}  \left[x \cdot \nabla u (x)\right] \frac{\partial w}{\partial \nu}(x) \, dS = 0.
\]
However, if $\Omega$ is either a ball or an annulus,  by the Hopf lemma, we infer that
\[
\left[x \cdot \nabla u (x)\right] \frac{\partial w}{\partial \nu}(x) > 0 \ \ \text{on} \ \ \partial \Omega \quad \text{or} \quad \left[x \cdot \nabla u (x)\right] \frac{\partial w}{\partial \nu}(x) < 0 \ \ \text{on} \ \ \partial \Omega, 
\]
since $u$ and $w$ have two and three nodal regions, respectively. Hence, the proof is complete.
\end{proof}

\section{Proof of Theorem \ref{maintheoremN=2}} \label{section:proofmaintheorem}
Let $\Omega \subset {\mathbb R}^2$ be either a ball or an annulus centered at the origin. Let $\alpha> 0$ and $f: {\mathbb R} \rightarrow {\mathbb R}$ be $C^{1, \beta}$ on bounded sets of ${\mathbb R}$. From now on we take $\kappa = \frac{2}{\alpha+2}$ as in \eqref{kappabom} and $\Omega_{\kappa} = T_{\kappa}^{-1}(\Omega)$, with $T_{\kappa}$ as in \eqref{tkappa}. Given $u \in H^1_0(\Omega)$ and $v \in H^1_0(\Omega_{\kappa})$, consider the bilinear forms
\[
Q_u(U,W) = \int_{\Omega} \nabla U \nabla W dx - \int_{\Omega} |x|^{\alpha} f'(u) UW dx, \quad U,W \in H^1_0(\Omega)
\]
and
\[
\mathcal{Q}_{v}(\mathcal{U}, \mathcal{W}) = \int_{\Omega_{\kappa}}\nabla \mathcal{U} \nabla \mathcal{W} dy - \left( \frac{2}{\alpha+ 2}\right)^2 \int_{\Omega_{\kappa}} f'(v)\mathcal{U} \mathcal{W} dy, \ \ \mathcal{U}, \mathcal{W} \in H^1_0(\Omega_{\kappa})
\]
associated with \eqref{weight+f} and \eqref{reduced}, respectively. The crucial point for the proof of Theorem \ref{maintheoremN=2} is the following result.

\begin{proposition}\label{comparisonlemma}
Let $v, \psi \in H^1_{0}(\Omega_{\kappa})$ and set $u = v \circ T_{\kappa}^{-1}$ and $\varphi = \psi \circ T_{\kappa}^{-1}$. Then
 \begin{equation}\label{generalenergy}
 \mathcal{Q}_v(\psi, \psi) \geq \frac{2}{\alpha+2} Q_u(\varphi, \varphi), \quad \forall \,  \psi \in H^1_{0}(\Omega_{\kappa})
 \end{equation}
and
 \begin{equation}\label{radialenergy}
 \mathcal{Q}_v(\psi, \psi) = \frac{2}{\alpha+2} Q_u(\varphi, \varphi), \quad \forall \,  \psi \in H^1_{0,{\rm{rad}}}(\Omega_{\kappa}).
 \end{equation}

\end{proposition}
\begin{proof} It is a direct consequence of Lemmas \ref{lemma:flrnorms} and \ref{lemma:h1}.
\end{proof}

\begin{proof}[\textbf{Proof of Theorem \ref{maintheoremN=2}}] Let $u$ be a radial nodal solution of \eqref{weight+f}. Then, define $v:\Omega_{\kappa} \rightarrow {\mathbb R}$ by setting $v(y) = u(T_{\kappa}(y))$, with $\kappa  = \frac{2}{\alpha+ 2}$. Hence $v$ is a radial nodal solution of \eqref{reduced}. Observe that the  eigenvalue problem for the linearized operator associated with \eqref{reduced} is
\begin{equation}\label{linearizedreduced}
 -\Delta \psi - \left(\frac{2}{\alpha+2}\right)^2 f'(v) \psi = \lambda \psi \quad \text{in} \quad \Omega_{\kappa}, \qquad
 \psi = 0 \quad \text{on} \quad \partial \Omega_{\kappa}.
\end{equation}
Hence, if $\psi$ is a radial eigenfunction of \eqref{linearizedreduced} then, writing $\psi(s) = \varphi(s^{\frac{2}{\alpha+2}})$,  we infer from \eqref{laplacians} and \eqref{kappabom} that $\varphi$ is a radial eigenfunction of
\begin{equation}\label{linearized}
 -\Delta \varphi - |x|^{\alpha} f'(u) \varphi = \lambda  \left(  \frac{\alpha+2}{2}\right)^{2} |x|^{\alpha} \varphi \ \ \text{in} \ \ \Omega, \qquad
 \varphi = 0 \quad \text{on} \quad \partial \Omega.
\end{equation}

We know, from \cite{AftalionPacella}, that the Morse index of $v$ is at least $3$ and greater than or equal to $n(u)+2$ if \eqref{H:superlinear} is satisfied; cf. Theorem A and Remark \ref{remark4int} in the introduction. More precisely, the problem \eqref{linearizedreduced} has a negative eigenvalue $\lambda_{1,{\rm{rad}}}$ (the first eiganvalue) with a corresponding radial eigenfunction $\psi_{1,{\rm{rad}}}$ and there are two other negative eigenvalues $\lambda_2=\lambda_3$ with corresponding eigenfunctions $\psi_2$ and $\psi_3$. Moreover, see \cite{AftalionPacella},
\begin{equation}\label{oddevenaftalionpacella}
\begin{array}{l}
\psi_2(y_1, y_2) \ \ \text{is even w.r.t.} \ \ y_2 \ \ \text{and} \ \ \text{odd w.r.t.} \ \ y_1,\\
\psi_3(y_1, y_2) \ \ \text{is even w.r.t.} \ \ y_1 \ \ \text{and} \ \ \text{odd w.r.t.} \ \ y_2.\\
\end{array}
\end{equation}
Hence, in particular, 
\[
\mathcal{Q}_v(\psi_{1, {\rm{rad}}}, \psi_{1, {\rm{rad}}}) < 0 \ \ \text{and} \ \ \mathcal{Q}_v(\psi_{i}, \psi_{i}) < 0, \quad  \ \ i=2,3.
\]
Moreover, if \eqref{H:superlinear} is satisfied then the radial eigenvalues of \eqref{linearizedreduced}, up to the $n(u)$-th, are also negative. In this case let us denote these eigenvalues by $\lambda_{i, {\rm{rad}}}$ and the associated radial eigenfunctions by $\psi_{i, {\rm{rad}}}$, $i=2, \ldots, n(u)$.

As we have observed, the change of variable $s \mapsto s^{\kappa}$, guarantees that $\varphi_{i, {\rm{rad}}}$, defined by $\psi_{i, {\rm{rad}}} (y) = \varphi_{i, {\rm{rad}}}(T_{\kappa}(y))$ with $i=1,2, \ldots, n(u)$, are radial eigenfunction of \eqref{linearized} with $\lambda = \lambda_{i, {\rm{rad}}}$. Eventhough, $\varphi_2$ and $\varphi_3$ defined by $\varphi_{i}(x) = \psi_{i} (T_{\kappa}^{-1}x)$, $i=2,3$, are not eigenfunctions of \eqref{linearized}, they correspond to directions in which the quadratic form induced by $Q_u$ is negative definite, which follows from \eqref{generalenergy}.  Moreover, using that
\begin{enumerate}[i)] 
\item $\varphi_{i, {\rm{rad}}}$, $i=1,2, \ldots, n(u)$, are eigenfunctions of \eqref{linearized} with $\lambda = \lambda_{i, {\rm{rad}}}$;
\item the symmetries of $\varphi_{1, {\rm{rad}}}, \varphi_{2, {\rm{rad}}}, \ldots, \varphi_{n(u),{\rm{rad}}}, \varphi_2, \varphi_3$; 
\end{enumerate}
it is simple to verify that  $\varphi_{1, {\rm{rad}}}, \varphi_{2, {\rm{rad}}}, \ldots, \varphi_{n(u),{\rm{rad}}}, \varphi_2, \varphi_3$ are mutually orthogonal with respect to both the bilinear forms
\[
\begin{array}{l}
(U,W) \mapsto \int_{\Omega} |x|^{\alpha} U W dx, \quad \text{and}\\ \\
(U,W) \mapsto Q_u(U,W) = \int_{\Omega} \left[ \nabla U \nabla W - |x|^{\alpha} f'(u) UW \right]dx.
\end{array}
\]
Therefore, we infer that $Q_u(w,w) < 0$ for every nonzero $w$ in the span $\left[\varphi_{1, {\rm{rad}}}, \varphi_2, \varphi_3  \right]$ or for every nonzero $w$ in the span $\left[\varphi_{1, {\rm{rad}}}, \varphi_{2, {\rm{rad}}}, \ldots, \varphi_{n(u),{\rm{rad}}}, \varphi_2, \varphi_3 \right]$ if \eqref{H:superlinear} is satisfied. This proves Theorem \ref{maintheoremN=2}. \qedhere
\medbreak
\end{proof}

\section{Other changes of variables: proof of Theorem \ref{theoremalpha=2}}  \label{section:doublysymmetric}
To the aim of proving Theorem \ref{theoremalpha=2} we now consider a variant of the change of variable in ${\mathbb R}^2$ defined in Section \ref{section:generalchangeofvariables}, which involves changing both polar coordinates $r$ and $\theta$. 

Given $\kappa>0$ and $m \in {\mathbb N}$ we set
\begin{equation}\label{tkappam}
\begin{array}{l}
T_{\kappa,m}: [0, \infty) \times [0, 2\pi] \rightarrow [0, \infty) \times [0, \frac{2\pi}{m}], \vspace{6pt}\\
T_{\kappa,m}(s, \sigma) := \left(s^{\kappa}, \frac{\sigma}{m}\right), \ \ r = s^{\kappa}, \ \ \theta = \frac{\sigma}{m}.
\end{array}
\end{equation}
Obviously $T_{\kappa,1}$ is just $T_{\kappa}$ of \eqref{tkappapolar}.

Consider any continuous function $\psi$ defined on a radially symmetric domain $\Omega$ in ${\mathbb R}^2$ in the cartesian coordinates $(y_1,y_2)$. Then, as in Section \ref{section:generalchangeofvariables}, using the polar coordinates
\[
y_1 = s \cos \sigma, \ \ y_2 = s \sin\sigma, \ \ s= \sqrt{|y_1|^2 + |y_2|^2},
\]
we can write
\[
\psi(y_1, y_2) 
= \psi(s \cos \sigma, s \sin \sigma) = \psi (s, \sigma), \ \  \text{with} \ \ \sigma \in [0, 2\pi] \ \ \text{and} \ \ \psi(s,0) = \psi(s, 2\pi).
\]
We then set
\begin{equation}\label{psivarphim}
\varphi(x_1, x_2) = \varphi(r, \theta) = \psi(T_{\kappa,m}^{-1}(r, \theta)).
\end{equation}
Hence $\varphi$ is a function defined for $\theta \in [0, \frac{2\pi}{m}]$ which, since $\psi(s, 0) = \psi(s, 2 \pi)$, can be extended $\frac{2\pi}{m}$-periodically and continuously for all $\theta \in [0, 2\pi]$. We still denote this extension by $\varphi$ and we observe that if it is smooth, by direct computation, then we have
\begin{equation}\label{laplacianspolarkappam}
\kappa^2 r^{2 - \frac{2}{\kappa}}\left[\varphi_{rr} + \frac{1}{r} \varphi_{r} + \frac{1}{r^2} \varphi_{\theta \theta}\right] =  \psi_{ss} + \frac{1}{s} \varphi_{s} + \frac{m^2 \kappa^2}{s^2} \psi_{\sigma \sigma}.
\end{equation}
Hence if we choose $\kappa = \frac{1}{m}$, for the Laplacian in cartesian coordinates we have 
\begin{equation}\label{fundamental2m}
m^{-2} |x|^{2( 1-m)} \Delta \varphi (x) = \Delta \psi(y).
\end{equation}

In view of the relation \eqref{fundamental2m} involving the Laplacians of $\varphi$ and $\psi$, we will apply the above procedure to work with the Hénon type equations \eqref{weight+f} in the case that $\alpha = 2 (m-1)$, with $m \geq 2$, that is for every $\alpha$ even. Indeed 
\begin{equation}\label{4.3'}
\alpha = 2 (m-1) \Longleftrightarrow \kappa  = \frac{1}{m} = \frac{2}{\alpha+2}
\end{equation}
which coincides with the relation \eqref{kappabom} between $\kappa$ and $\alpha$.

Note that, in view of the complex plane, the above transformation $T_{\frac{1}{m},m}$ is just the one which sends $z$ into $z^{\frac{1}{m}}$, $z \in {\mathbb C}$.

\begin{remark}\label{remarkpacellasrikanth}
Observe that, in the particular case when $m$ is even, if $\psi$ is a function such that
\[
\psi(y_1, y_2) = \psi(y_1, -y_2), \ \ (y_1, y_2) \in {\mathbb R}^2
\]
i.e. even with respect to $y_2$, then the extended function $\varphi(x_1, x_2)$, given by $\psi = \varphi \circ T_{\frac{1}{m},m}$, is such that $\varphi$ is even with respect to $x_1$ and $x_2$, that is
\[
\varphi(x_1, x_2) = \varphi(|x_1|, |x_2|), \ \ (x_1, x_2) \in {\mathbb R}^2.
\]
Hence functions that are symmetric with respect to one axis produce functions that are symmetric with respect to both axes.
\medbreak
\end{remark}

With the above choice of $\alpha$ we consider a radial nodal solution $u$ of \eqref{weight+f}. By Theorem \ref{maintheoremN=2} we know that $u$ has Morse index greater than or equal to $3$ and at least $n(u)+2$ if \eqref{H:superlinear} is also satisfied. We will use the change of variable \eqref{tkappam} with $\kappa = \frac{1}{m}$ to construct $\alpha + 2 = 2m$ convenient non-radial directions on which the quadratic form $Q_u(w,w)$ is negative. 

We can now proceed with the proof of Theorem \ref{theoremalpha=2}.

\begin{proof}[\textbf{Proof of Theorem \ref{theoremalpha=2}}]
Let $\alpha = 2(m-1)$, with $m \geq 2$, $\kappa = \frac{1}{m}$, and let $u$ be a radial nodal solution of \eqref{weight+f}. Then, by \eqref{reduced}, the radial function $v = u \circ T_{\kappa}$ solves 
\[
-\Delta v = \frac{1}{m^2} f(v) \quad \text{in} \quad \Omega_{\kappa}, \quad v = 0 \quad \text{on} \quad \partial \Omega_{\kappa}.
\]
Therefore, by the results of \cite{AftalionPacella}, already used at \eqref{oddevenaftalionpacella}, there exist two eigenfunctions $\psi_2$ and $\psi_3$ for the eigenvalue problem
\begin{equation}\label{4.3''}
-\Delta \psi  - \frac{1}{m^2} f'(v) \psi = \lambda \psi \quad \text{in}  \quad \Omega_{\kappa}, \quad \psi = 0 \quad \text{on} \quad \partial \Omega_{\kappa},
\end{equation}
with the following properties:
\begin{enumerate}[i)]
\item the corresponding eigenvalues $\lambda_2 = \lambda_3$ are negative;
\item $\psi_2$ is even with respect to $y_2$ and odd with respect to $y_1$, while $\psi_3$ is even with respect to $y_1$ and odd with respect to $y_2$;
\item $\psi_2 (y_1, y_2) > 0$ if $y_1>0$, while $\psi_3(y_1, y_2) > 0$ if $y_2 >0$.
\end{enumerate}

Next, applying the change of variables \eqref{tkappam}, we consider the functions $\varphi_{m,i}(r, \theta) = \psi_i \circ T_{\frac{1}{m},m}^{-1}(r, \theta)$, $i=2,3$, extended by periodicity as before for all $\theta \in [0, 2 \pi]$, so to have them defined on the whole $\Omega$. Then, by the conditions $\frac{\partial \psi_2}{\partial \sigma} = 0$ and $\psi_3 = 0$ at $\sigma=0$, we have that $\varphi_{m, i}$, $i=2,3$, are $C^2(\overline{\Omega})$-functions and by \eqref{fundamental2m} they satisfy
\begin{equation}\label{fiftheigenvalue}
-\Delta \varphi  - |x|^{\alpha}f'(u) \varphi =  \lambda |x|^{\alpha}\varphi \quad \text{in}  \quad \Omega, \quad \varphi = 0 \quad \text{on} \quad \partial \Omega,
\end{equation}
with $\lambda =\lambda_i  m^2$. Moreover it is easy to see that both $\varphi_{m,i}$, $i=2,3$, have $2m$ nodal sets, each one being an angular sector of amplitude $\frac{\pi}{m}$. This means that each one is a first eigenfunction of \eqref{fiftheigenvalue} in that sector with corresponding eigenvalue $\lambda_i m^2 < 0$. In particular $\varphi_{m, 2}$ is the first eigenfunction in the sector
\[
\Omega_{m,2} = \left\{ (x_1, x_2) = (r \cos \theta, r \sin \theta) \in \Omega, \theta \in \left[- \frac{\pi}{2m}, \frac{\pi}{2m}\right] \right\},
\]
while $\varphi_{m,3}$ is the first eigenfunction in the sector
\[
\Omega_{m,3} = \left\{ (x_1, x_2) = (r \cos \theta, r \sin \theta) \in \Omega, \theta \in \left[0, \frac{\pi}{m}\right] \right\}.
\]
Then, by the monotonicity of the first eigenvalues with respect to the domain, by inclusion, we have that the first eigenvalue in $\Omega_{n,2}$ or $\Omega_{n,3}$ are also negative for every integer $1 \leq n < m$, $\Omega_{n,i}$ defined as before, replacing $m$ by $n$, for $i=2,3$. The corresponding eigenfunctions, say $\varphi_{n,i}$ extended by oddness with respect to the anticlockwise border of $\Omega_{n,i}$ and periodically, with angular period $\frac{2\pi}{n}$, give rise to other two eigenfunctions for \eqref{fiftheigenvalue}, for every $n \in \{1, \ldots, m\}$. By construction, their symmetry or antisymmetry, all these pairs of eigenfunctions are mutually orthogonal with respect to both the bilinear forms  
\begin{equation}\label{bilinearforms}
\begin{array}{l}
(U,W) \mapsto \int_{\Omega} |x|^{\alpha} U W dx, \quad \text{and}\\ \\
(U,W) \mapsto Q_u(U,W) =\int_{\Omega} \left[ \nabla U \nabla W - |x|^{\alpha} f'(u) UW \right]dx,
\end{array}
\end{equation}
so that we get $2m$ negative eigenvalues for \eqref{fiftheigenvalue} corresponding to nonradial directions. Counting also the first radial eigenvalue, which is negative, and from the second up to the $n(u)$-th radial eigenvalue which are also negative if \eqref{H:superlinear} holds, we get the assertion, since $\alpha = 2 (m-1)$.
\end{proof}

\subsection*{Acknowledgements} The authors thank D. Bonheure and H. Tavares for some interesting discussions on the subject of this paper. This work was done while the first author was visiting the Dipartimento di Matematica of the Università di Roma {\it{Sapienza}}, whose hospitality he gratefully acknowledges.

\end{document}